\documentclass[10pt]{article}

\AtEndDocument{\bigskip{\footnotesize
     \textsc{Marcelo Miranda, IME-USP, São Paulo, Brazil.} \par  
  \textit{E-mail}:  \texttt{marcelojosecf@gmail.com}\\

  \textsc{Brayan Ferreira, Universidade Federal do Espírito Santo, Vitória, Brazil.} \par  
  \textit{E-mail}:  \texttt{brayan.ferreira@ufes.br}\\
  
   \textsc{Alejandro Vicente, University of Stavanger, Stavanger, Norway.} \par  
  \textit{E-mail}:  \texttt{kvicente931207@gmail.com}
  }}

\title{Canonical frames in contact 3-manifolds and applications}
\author{Brayan Ferreira, Marcelo Miranda, Alejandro Vicente}
\date{}
\usepackage[a4paper, total={6in, 8in}]{geometry}
\usepackage{amssymb}
\usepackage{latexsym,hyperref}

\usepackage{enumitem}
\usepackage{amsmath}
\usepackage{amsthm}
\usepackage{amscd}
\usepackage{color}
\usepackage{graphicx}
\usepackage{tikz-cd}
\usepackage{nicefrac}
\usepackage{tikz}
\usepackage{enumitem}
\usepackage{pgfplots}
\usepackage{tikz-3dplot}
\pgfplotsset{compat=1.18}

\numberwithin{equation}{section}

\newtheorem{theorem}{Theorem}[section]

\newtheorem{proposition}[theorem]{Proposition}
\newtheorem{corollary}[theorem]{Corollary}
\newtheorem{lemma}[theorem]{Lemma}
\newtheorem{lemma-definition}[theorem]{Lemma-Definition}

\theoremstyle{definition}

\newtheorem{remark}[theorem]{Remark}

\newtheorem{example}[theorem]{Example}

\newtheoremstyle{italicclaim}
  {}{}
  {\itshape}
  {}
  {\itshape}
  {.}
  { }
  {}

\theoremstyle{italicclaim}

\newcommand{\R}{\mathbb{R}}
\newcommand{\C}{\mathbb{C}}
\newcommand{\Z}{\mathbb{Z}}
\renewcommand{\epsilon}{\varepsilon}

\begin{document}
\maketitle

\begin{abstract}
We study contact 3-manifolds $Y$ with a special global frame inspired by Cartan's structure equations. This frame is dual to a generalized Finsler structure defined by Bryant. We present some examples and rigidity results on the class of manifolds whose frame satisfies certain natural conditions on a scalar function $K\colon Y\to \mathbb{R}$, related to the frame. This function realizes the curvature when $Y$ is the unit tangent bundle with respect to a metric on a surface. As applications, we obtain sharp estimates for the action of a Reeb orbit in terms of this scalar function, under the assumption that the frame satisfies specific conditions. In particular, we recover a classical upper bound on the systole of positively curved metrics on $S^2$ due to Toponogov.

\end{abstract}

\section{Introduction}\label{sec: intro}

Let $\Sigma$ be a closed oriented connected surface. A Finsler metric $F$ on $\Sigma$ is a smooth, positive 1-homogeneous, strictly convex function on $T\Sigma\setminus \{0\}$, that extends continuously to the zero section of $T\Sigma$. We define the smooth hypersurface $S\Sigma \subset T\Sigma$ given by $S\Sigma:=F^{-1}(1)$, for which the canonical projection $\pi : S\Sigma \to \Sigma$ is a surjective submersion having the property that for each $\mathbf{q} \in \Sigma$, the fiber $\pi^{-1}(\mathbf{q}) = S\Sigma \cap T_\mathbf{q}\Sigma$ is a smooth, closed, strictly convex curve enclosing the origin $0_\mathbf{q} \in T_\mathbf{q}\Sigma$. Given such a structure, following Cartan, it is possible to define a canonical (global) coframing
$(\omega_1, \omega_2,\omega_3)$ on $S\Sigma$ that satisfies the following structural equations; see \cite[Chapter 4]{BaoChernShen2000}
\begin{equation}\label{eq: structure_eqs}
   \begin{aligned}
   d\omega_1&=-\omega_2\wedge\omega_3,\\
       d\omega_2&=-\omega_3\wedge(\omega_1-I\omega_2),\\ 
       d\omega_3&=-(K\omega_1-J\omega_3)\wedge\omega_2, 
       \end{aligned}
\end{equation} 
where $I, K$ and $J$ are smooth functions on $S\Sigma$. The scalar $K$ is the $\pi$-pullback of the Gaussian
curvature whenever $F$ defines the norm associated to a Riemannian metric, and the vanishing of $I,J$ and $K$, respectively characterize Riemannian surfaces, Landsberg surfaces and flat surfaces; see \cite{BaoChernShen2000}. In the Finsler setting, the function $I$ is called the \textit{Cartan (\textup{or} main) scalar}, the function $J$ is called the \textit{Landsberg scalar}  and the function $K$ is called the \textit{Gaussian curvature}.\\
Let $R, X_1$ and $X_2$ be the vector fields on $S\Sigma$ that are dual to the
coframing $(\omega_1, \omega_2, \omega_3)$. The form $\omega_1$ is the Hilbert contact form of $S\Sigma$
whose Reeb vector field $R$ coincides with the geodesic vector field, see \cite{hryniewicz2013introduccao} or \cite{dorner2017finsler}. As a consequence of \eqref{eq: structure_eqs} the framing $(R, X_1, X_2)$ satisfies the Lie bracket relations 
\begin{equation}\label{eq: frame_def}
    \left[ X_2,R \right]=X_1,\quad \left[X_1,X_2 \right]=R+IX_1+JX_2,\quad \left[R,X_1\right]=KX_2.
\end{equation}
We refer to this frame as a \textit{canonical frame}.

In this note we want to study more general contact manifolds, other than the unit tangent bundles of surfaces, that admit global canonical frames. In order to do so, we consider the following class of manifolds. Let $\lambda$ be a contact form in a closed 3-manifold $Y$ and $R$ its Reeb vector field. We assume that there exist $X_1$ and $X_2$ sections of $\xi=\ker\lambda$, and smooth functions $I,J,K:Y\rightarrow \mathbb{R}$ satisfying equations \eqref{eq: structure_eqs}. We call $\{R,X_1,X_2\}$ a \textit{contact canonical frame} for the contact form $\lambda$. We note that this definition is dual to that of a \emph{generalized Finsler structure}, a notion introduced by Bryant in \cite{bryant1996finsler,bryant1997projectively}, see also \cite{sabau2014generalized}. 

That such a frame always exists locally can be argued by a simple argument as follows. Indeed, for $\mathbb{R}^3$ with the standard contact structure induced by the contact form $\lambda=dz-ydx$, we have that $R=\partial_z$ and the frame given by $X_1=\partial_x-y\partial_z$ and $X_2=\partial_y-z\partial_x+yz\partial_z$ satisfies the structure equations in \eqref{eq: frame_def} for $I(x,y,z)=y$ and $J=K=0$, see Figure \ref{fig: frame}. Then, a standard argument with Darboux Theorem proves local existence.

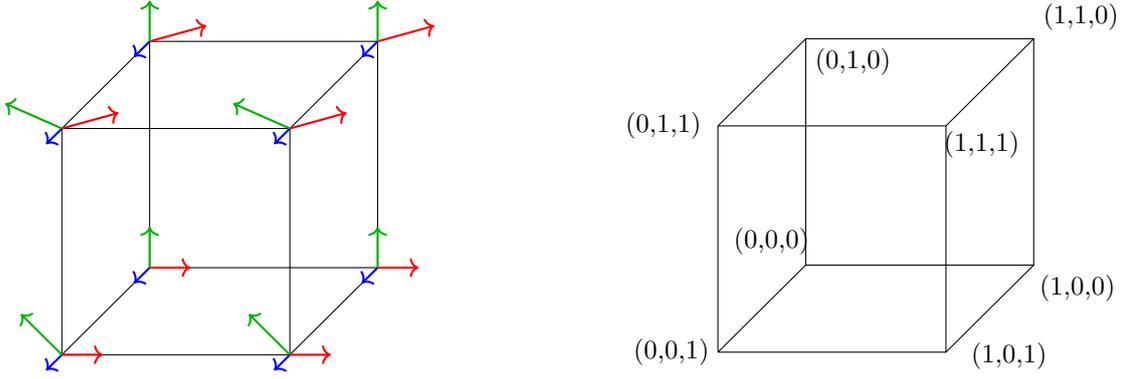
\begin{figure}[ht]
\centering

\begin{minipage}{0.48\textwidth}
\centering
\begin{tikzpicture}[scale=3]

\def\s{0.18}

\draw[black] (0,0,0) -- (1,0,0) -- (1,1,0) -- (0,1,0) -- (0,0,0);
\draw[black] (0,0,1) -- (1,0,1) -- (1,1,1) -- (0,1,1) -- (0,0,1);
\draw[black] (0,0,0) -- (0,0,1);
\draw[black] (1,0,0) -- (1,0,1);
\draw[black] (1,1,0) -- (1,1,1);
\draw[black] (0,1,0) -- (0,1,1);

\foreach \x in {0,1} {
\foreach \y in {0,1} {
\foreach \z in {0,1} {

\draw[->, blue, thick] (\x,\y,\z) -- (\x,\y,{\z+\s});
\draw[->, red, thick] (\x,\y,\z) -- ({\x+\s},{\y},{\z-\s*\y});
\draw[->, green!70!black, thick] (\x,\y,\z) -- ({\x-\s*\z},{\y+\s},{\z+\s*\y*\z});
}}}

\end{tikzpicture}
\end{minipage}
\hfill
\begin{minipage}{0.48\textwidth}
\centering
\begin{tikzpicture}[scale=3]

\draw[black] (0,0,0) -- (1,0,0) -- (1,1,0) -- (0,1,0) -- (0,0,0);
\draw[black] (0,0,1) -- (1,0,1) -- (1,1,1) -- (0,1,1) -- (0,0,1);
\draw[black] (0,0,0) -- (0,0,1);
\draw[black] (1,0,0) -- (1,0,1);
\draw[black] (1,1,0) -- (1,1,1);
\draw[black] (0,1,0) -- (0,1,1);

\node[anchor=north east] at (0,0.15,-0.15) {(0,0,0)};
\node[anchor=south east] at (0.1,0,1.25) {(0,0,1)};
\node[anchor=north west] at (0,1,0) {(0,1,0)};
\node[anchor=south west] at (-0.35,1,1.25) {(0,1,1)};
\node[anchor=north east] at (1.4,0,0) {(1,0,0)};
\node[anchor=south east] at (1.6,0,1.3) {(1,0,1)};
\node[anchor=north west] at (1,1.2,0) {(1,1,0)};
\node[anchor=south west] at (1,0.87,1.13) {(1,1,1)};

\end{tikzpicture}
\end{minipage}

\caption{
Frame $(R,X_1,X_2)$ for $\mathbb{R}^3$ in blue, red and green, respectively.
}
\label{fig: frame}
\end{figure}

Note that a contact canonical frame induces an almost complex structure $\mathbb{J}$ on the plane bundle $\xi\rightarrow Y$, defined by
\begin{equation}\label{def: almostcpx}
\mathbb{J}(aX_1+bX_2)=bX_1-aX_2.
\end{equation}
Moreover, observe that for $Z=aX_1+bX_2$, we have
$$
d\lambda(Z,\mathbb{J}Z)=d\lambda(aX_1+bX_2,bX_1-aX_2)=-a^2d\lambda(X_1,X_2)+b^2d\lambda(X_2,X_1)=(a^2+b^2)d\lambda(X_2,X_1).
$$
Since
$$
d\lambda(X_2,X_1)=X_2(\lambda(X_1))-X_1(\lambda(X_2))-\lambda(\left[X_2,X_1\right])=\lambda(\left[X_1,X_2\right])=1,
$$
we obtain
\begin{equation}\label{eq: dlambda}
d\lambda(Z,\mathbb{J}Z)=a^2+b^2>0.    
\end{equation}
Thus $\mathbb{J}$ is a compatible almost complex structure. We can then define a Riemannian metric  $\langle\cdot,\cdot\rangle$ on $Y$ by considering the inner product in the contact structure $\xi$ induced by $d\lambda$ and $\mathbb{J}$ and extending to the whole manifold by declaring the Reeb vector field $R$ to be orthonormal to the contact structure $\xi$, that is
\begin{equation}\label{eq: riem_metric}
\langle\cdot,\cdot\rangle=\lambda\otimes\lambda+d\lambda(\cdot,\mathbb{J}\cdot).
\end{equation}

\subsection{Existence and classification results}

A class of contact 3-manifolds of remarkable importance is that of the boundaries of star-shaped toric domains. A \textit{four-dimensional toric domain} is a subset of $\C^2$ defined by
\[\mathbb{X}_\Omega=\left\{(z_1,z_2)\in \C^2\mid(\pi|z_1|^2,\pi|z_2|^2)\in\Omega\right\},\] where $\Omega\subset\R_{\ge 0}^2$ is the closure of an open set in $\R^2$. If the region $\Omega$ is star-shaped with respect to the origin in $\mathbb{R}^2_{\geq 0}$, then the toric domain $\mathbb{X}_{\Omega}$ is also a star-shaped region in $\mathbb{C}^2$, and therefore the boundary $\partial \mathbb{X}_{\Omega}$ is a contact manifold, with contact form given by contracting the standard symplectic form in $\mathbb{C}^2$ with the radial Liouville vector field. In this case, we denote this contact form by $\lambda_{\mathrm{std}}$ and note that $\lambda_{\mathrm{std}} = 2\lambda_0$, where $\lambda_0$ is the standard Liouville form given by the restriction to $\partial \mathbb{X}_\Omega$ of the $1$-form
$\frac{1}{2} \sum_{j=1}^2 r_j^2 \, d\theta_j$
on $\mathbb{C}^2$, written in polar coordinates $z_j = r_j e^{i\theta_j}$. Furthermore, it is not hard to see that $\partial\mathbb{X}_{\Omega}$ is diffeomorphic to $S^3$. Notice for example, that if we take $\Omega\subset \mathbb{R}^2_{\geq 0}$ to be the triangle with vertices at $(0,0), (a,0)$ and $(0,b)$, for $a,b>0$, we have that $\mathbb{X}_{\Omega}=E(a,b)$, where $E(a,b)$ is the symplectic ellipsoid, i.e. 
\begin{equation}\label{eq:defellipsoid}
    E(a,b)=\bigg\{(z_1,z_2)\in \mathbb{C}^2\mid \pi\left(\frac{|z_1|^2}{a}+\frac{|z_2|^2}{b}\right)\leq 1\bigg\}.
\end{equation}



In the case of Finsler structures with global canonical coframes, as discussed in \eqref{eq: structure_eqs}, there is an interpretation of the scalar $K$ as Gaussian curvature in the Riemannian subcase. This outlines the importance of understanding canonical frames, and more generally contact canonical frames, with constant values of the scalar $K$, as a generalization of constant curvature surfaces. We begin by studying classes of contact 3-manifolds admitting a contact canonical frame where $K\equiv 1$, $K\equiv 0$ and $K\equiv -1$.

\begin{remark}\label{rmk: constantK}
   Let $Y$ be a closed contact $3$-manifold equipped with a contact form $\lambda$ with Reeb vector field $R$. Assume that $(Y,\lambda)$ admits a contact canonical frame $\{R,X_1,X_2\}$ such that the function $K$ is a nonzero constant, say $K=\pm k^2$ for $k> 0$. Then $\widetilde{\lambda} = k\lambda$ is a new contact form on $Y$ such that the Reeb vector field is given by $\widetilde{R} = \frac{1}{k} R$. In this case, $\{\frac{1}{k}X_1,X_2\}$ is a contact canonical frame for $\widetilde{\lambda}$ such that
    $$\begin{cases}[X_2,\widetilde{R}] = \frac{1}{k}X_1 \\ [\frac{1}{k}X_1, X_2] = \widetilde{R} + I\left(\frac{1}{k}X_1\right) + \frac{1}{k}JX_2 \\ [\widetilde{R},\frac{1}{k}X_1] = \textup{sgn}(K) X_2.
    \end{cases}$$
    where $\textup{sgn}(K)$ denotes the sign of the non-zero constant function $K$. In particular, $\{\widetilde{R},\frac{1}{k}X_1,X_2\}$ is a contact canonical frame for $\widetilde{\lambda}$ with functions $\widetilde{I},\widetilde{J},\widetilde{K}$ given by $\widetilde{I} = I$, $\widetilde{J} = \frac{1}{k}J$ and $\widetilde{K} \equiv \textup{sgn}(K)\cdot1$.
\end{remark}

We start with the case $K\equiv 0$.

\begin{theorem}\label{thm: K=0}
    Let $(Y,\lambda)$ be a closed contact 3-manifold and $\{R,X_1, X_2\}$ a contact canonical frame. If $K\equiv 0$, then
   \begin{enumerate}[label=(\roman*)]
    \item \label{it: I=J=0} $J\equiv 0$.
    \item \label{it: T^2-bundle} $Y$ is a $T^2$-bundle over $S^1$.
    \item \label{it: induced_I} The function $I$ induces a function $\tilde{I}:S^1\to \mathbb{R}$ in the base of the fibration in \ref{it: T^2-bundle} satisfying that 
    $$\int_{S^1}\tilde{I}d\theta=0.$$
\end{enumerate}
\end{theorem}

We furthermore give a complete classification if we assume in addition that $I\equiv 0$.

\begin{theorem}\label{thm: K=I=0}
     Let $(Y,\lambda)$ be a closed contact 3-manifold and $\{R,X_1, X_2\}$ a contact canonical frame. If $K\equiv I\equiv 0$, then
   \begin{enumerate}[label=(\roman*)]
    
    \item \label{it: flat_metric} the metric defined in \eqref{eq: riem_metric} for $Y$ is complete and flat .    
    \item \label{it: cryst_group} $Y$ is the quotient of $\mathbb{R}^3$ by a torsion-free crystallographic group. 
    \item \label{it: quot_of_R^3} $Y$ is finitely covered by $T^3$.
    \item \label{it: all_possibilities}
    $Y$ is isometric to exactly one of the following six types:

\begin{enumerate}
    \item $T^3$, the trivial bundle with monodromy given by $\phi = \mathrm{Id} $.
    
    \item Torus bundle with order 2 monodromy given by $\phi = -\mathrm{Id} = \begin{pmatrix}-1 & 0 \\ 0 & -1\end{pmatrix}$.
    
    \item Torus bundle with order 3 monodromy given by $\phi = \begin{pmatrix}0 & -1 \\ 1 & -1\end{pmatrix}$.
    
    \item Torus bundle with order 4 monodromy given by $\phi = \begin{pmatrix}0 & -1 \\ 1 & 0\end{pmatrix}$.
    
    \item Torus bundle with order 6 monodromy given by $\phi = \begin{pmatrix}0 & -1 \\ 1 & 1\end{pmatrix}$.
    
    \item Torus bundle with order 2 ``non-trivial twist''\footnote{This is known as the \textit{Hantzsche-Wendt flat 3-manifold} or \textit{didicosm}, see \cite{Hantzsche1935}.} given by $\phi = \begin{pmatrix}-1 & 0 \\ 1 & -1\end{pmatrix}$.
\end{enumerate}
\end{enumerate}
\end{theorem}

\begin{remark}
    We observe that the groups that give rise to the 6 manifolds in Theorem \ref{thm: K=0} are 6 out of the 10 possible Bieberbach groups in  dimension 3, the other 4 groups give rise, as a quotient in $T^3$, to non-orientable manifolds. See \cite{Hillman2002} for the complete list of this group. We would also like to observe that each of the five contact manifolds in items $(b)-(f)$ in Theorem \ref{thm: K=I=0}\ref{it: all_possibilities} carries the contact form obtained by taking the quotient of the contact manifold $T^3$, with contact form given by $\lambda=\cos\theta dx+\sin\theta dy$ (see Example \ref{ex: frame_T^3} for a detailed exposition of this case), under the appropriate deck transformation:
\begin{enumerate}[label=(\alph*), start=2]
    \item Order 2: quotient by $(\theta,y,z)\longmapsto(\theta+\pi,-y,-z)$,
    
    \item Order 3: quotient by
    $(\theta,y,z)\longmapsto\big(\theta+\tfrac{2\pi}{3},R_{2\pi/3}(y,z)\big),$  
    
    \item Order 4: quotient by $(\theta,y,z)\longmapsto\big(\theta+\tfrac{\pi}{2},R_{\pi/2}(y,z)\big)$,
    
    \item Order 6: quotient by $(\theta,y,z)\longmapsto\big(\theta+\tfrac{\pi}{3},R_{\pi/3}(y,z)\big)$,
    
    \item Order 2 (non-trivial twist): quotient by the $\mathbb{Z}_2 \times \mathbb{Z}_2$ group generated by the following two deck transformations: $(\theta, y, z) \mapsto (\theta + \pi, -y, -z)$ and $(\theta, y, z) \mapsto (-\theta, y + \pi, -z + \pi)$.
\end{enumerate}
Here $R_{\theta}$ denotes the rotation by angle $\theta\in [0,2\pi)$ in the $yz$-plane. 

\end{remark}

By Remark \ref{rmk: constantK}, we can focus on the cases where $K\equiv 1$ and $K\equiv -1$ to understand any case where $K$ is a nonzero constant. The topology of $Y$ admitting a contact canonical frame such that $K \equiv 1$ is discussed in \cite{sabau2014generalized}, namely, $Y$ must be a quotient of a Lie group $G$ by a discrete subgroup $\Gamma$, where $G$ is $S^3 \cong \mathrm{SU}(2)$, the universal covering of $\mathrm{PSL}_2(\R)$, $\widetilde{SL}_2$, or the group $\widetilde{E}_2$ of orientation preserving isometries of the plane. We give a nice classification in the case where $Y$ is a star-shaped toric domain with $K\equiv1$.

\begin{theorem}\label{thm: K=1}
    Let $\mathbb{X}_{\Omega}\subset \mathbb{R}^4$ be a smooth star-shaped toric domain such that $(\partial \mathbb{X}_{\Omega},\lambda_0)$ admits a contact canonical frame. Then $K \equiv 1$, if and only, if $X_\Omega = E(a,b)$, such that
    $$\pi(a+b) = \frac{ab}{2}.$$
\end{theorem}

The proof of Theorem \ref{thm: K=1} uses Fourier series expansion and the smoothness of the region $\Omega \subset \mathbb{R}^2_{\geq 0}$ to ensure that the boundary of $\Omega$ is given by the graph of a decreasing linear function. The characterization then follows from an explicit computation of the canonical contact frame in the case of the unit tangent bundle endowed with the Katok Finsler metric on $S^2$ (see Example \ref{ex:framekatok}), together with the good behavior of canonical contact frames under finite covering maps (see Lemma \ref{lemma:covering}).



\begin{remark}
    It follows from \cite[Lemma 3.1]{harris2008dynamically} that $\mathcal{L}_R{\mathbb{J}}=0$ holds, if and only if, $K\equiv 1$ for any closed contact $3$-manifold $(Y,\lambda)$ assuming a contact canonical frame. In particular, the Reeb vector field $R$ is Killing, if and only if, $K\equiv 1$. So this is the unique case where $(Y,\lambda,\langle,\rangle)$ is a $K$-contact manifold.
\end{remark}

Finally, we show that if $K=-1$, then $Y$ is the total space of a circle bundle over an orbifold with negative Euler characteristic, as in the case of the unit tangent bundle over a hyperbolic surface $\Sigma_g$.

\begin{theorem}\label{thm: K=-1}
 Let $(Y,\lambda)$ be a closed contact 3-manifold and $\{R,X_1, X_2\}$ a contact canonical frame. If $K\equiv -1$, then $I \equiv J \equiv 0$ and $Y$ is the total space of a circle bundle over a $2$-dimensional orbifold $X$ with negative orbifold Euler characteristic.
 \end{theorem}

\subsection{Landsberg contact manifolds and unit tangent bundles}
We say that $(Y,\lambda)$ is a \emph{Landsberg contact manifold} if there exists a contact canonical frame $\{R,X_1,X_2\}$ in the sense of equations \eqref{eq: frame_def} such that $J \equiv 0$. Example \ref{examp: sphere} confirms that $(S^3,\lambda_{std})$ is a Landsberg contact manifold. It follows from Theorem \ref{thm: K=0} and Theorem \ref{thm: K=-1} that the conditions $K \equiv 0$ or $K\equiv -1$ imply the Landsberg condition. A straightforward consequence from Proposition \ref{eq: prop: Jacobi} below, is that $(Y,\lambda)$ is Landsberg whenever $I$ is constant along the Reeb flow. In particular, unit tangent bundles of Riemannian metrics over surfaces are examples of Landsberg contact manifolds. The next result describes a relation between the functions $K$ and $I$ in Landsberg contact manifolds.

\begin{theorem}\label{thm:landsbergctc}
    Let $(Y,\lambda)$ be a Landsberg contact manifold. Then the function $K(t) := K \circ \psi(t)$, where $\psi$ is a flowline of the vector field $X_2$, satisfies
    \begin{equation}\label{eq:curvatureode}
        K(t) = K(0)e^{-\left[\int_0^{t}I(\tau) \ d\tau\right]}.
    \end{equation}
    In particular, if $y \colon \R/T\Z \to Y$ is a periodic orbit of $X_2$ such that $K(T) = K(0) \neq 0$, then
    $$\int_0^TI\circ y(\tau) \ d\tau = 0.$$
\end{theorem}

Note that equation \eqref{eq:curvatureode} implies that $K$ is constant along a flow line of $X_2$ where $I$ vanishes. In particular, when $Y$ is the unit tangent bundle of a Riemannian metric, we have $I \equiv 0$ and then $K$ is constant along any flow line of $X_2$. This is an expected behavior since $X_2$ generates the tangent fibers over the base in the sense of the following remark.

\begin{remark}
    Let $(Y,\lambda)$ be a contact manifold with a contact canonical frame $\{R,X_1,X_2\}$. Assume that $X_2$ is a periodic vector field such that its flow defines a free and smooth $S^1$ action on $Y$. The orbit space $\Sigma = Y/S^1$ is an orientable smooth closed surface since we have $T\Sigma \cong \xi$ and $d\lambda|_\xi$ is nondegenerate. Moreover, the natural projection $\pi\colon Y \to \Sigma$ is a smooth submersion. In this case, one can define a smooth map $\nu \colon Y \to T\Sigma$ setting $\nu(y) = d\pi_y(R(y))$ which turns out to be an immersion to each fiber $\pi^{-1}(x)$ as a closed and strictly convex curve in $T_x\Sigma$ enclosing $0 \in T_x\Sigma$. In particular, when $\nu$ is injective, $Y$ can be interpreted exactly as the unit tangent bundle for a Finsler metric $F$ on $\Sigma$ such that $F^{-1}(1) = \bigcup_{x \in \Sigma} \nu(\pi^{-1}(x))$. The converse also holds, i.e., if $Y = S\Sigma$ is the unit tangent bundle over a compact surface then $X_2$ is periodic, the orbit space $Y/S^1$ is diffeomorphic to $\Sigma$ and the projection $Y \to Y/S^1$ is a smooth submersion; see the discussion in \cite[Section 1.2]{bryant1997projectively}.
\end{remark}

\subsection{Applications}\label{sec: applications1}

As a first application, we use the Sturm comparison theorem for second-order ODEs to obtain a lower bound on the action of (closed) Reeb orbits under additional hypotheses on the contact frame. Given a contact manifold $(Y,\lambda)$ and a Reeb trajectory $\gamma \colon [0,T] \to Y$, the \emph{action of $\gamma$} is the number $\mathcal{A}(\gamma) = \int_\gamma \lambda = T$.

When the function $K \colon Y \to \R$ is constant along the flow lines of $X_2$, the function $I$ satisfies a second-order ODE of the same type as those appearing in the Sturm comparison theorem when restricted to Reeb flow lines; see Proposition~\ref{eq: prop: Jacobi} and Theorem \ref{thm: sturm} below. We obtain the following result.

\begin{theorem}\label{thm:sturmcompar}
   Let $(Y,\lambda)$ be a closed contact $3$-manifold admitting a contact canonical frame. Suppose that $K$ is constant along the flow lines of the vector field $X_2$. Consider $I(t) = I \circ \varphi(t)$, $t \in \R$, where $\varphi$ is a Reeb flow line such that $I(t)$ is not identically zero. Assume also that $K(t) > 0$ for all $t$. Then, the zero set of $I(t)$ is an infinite and discrete subset of $\R$, and if $t_1$ and $t_2$ are two consecutive zeros, we have
\[
    |t_1 - t_2| < \frac{2\pi}{\sqrt{\inf K \circ \varphi}}.
\]
Moreover, let $t_1 < t_2 < \cdots < t_n$ be real numbers such that $I(t_1) = I(t_2) = \cdots = I(t_n)$.
Then,
\[
    t_n - t_1 > \frac{(n-2)\pi}{\sqrt{\sup K \circ \varphi}}.
\]
In particular, if $\gamma \colon \R/T\Z \to Y$ is a Reeb orbit on $(Y,\lambda)$, then
\begin{equation}\label{eq:sturmbound}
    T = \mathcal{A}(\gamma) \ge \frac{\pi}{\sqrt{\max K \circ \gamma}}.
\end{equation}
\end{theorem}
Examples of contact $3$-manifolds satisfying the hypotheses of the theorem include the unit tangent bundles of spheres endowed with Katok Finsler metrics; see Example~\ref{ex:framekatok} below. In this setting, the minimal action among closed Reeb orbits is given by $2\pi/(1+\lvert a\rvert)$. In particular, this shows that the upper bound~\eqref{eq:sturmbound} is sharp: indeed, as $\lvert a\rvert \to 1$, the minimal action converges to $\pi$, while the curvature remains constant, $K \equiv 1$.

Following ideas on \cite{harris2008dynamically}, we also show that under the presence of a contact canonical frame, one can estimate the action of a Reeb orbit in terms of its Conley-Zehnder index and how the function $K$ is bounded along the orbit. In the following, we denote by $CZ(\gamma)$ the Conley--Zehnder index of a (closed) Reeb orbit $\gamma$ with respect to the trivialization of the contact structure $\xi$ given by a contact canonical frame $\{R,X_1,X_2\}$; see section \ref{sec: CZ_index}. Here, $\lfloor x \rfloor$ denotes the floor of the real number $x$, i.e., the greatest integer less than or equal to $x$.

\begin{theorem}\label{thm:estimate}
    Let $(Y,\lambda)$ be a closed contact $3$-manifold admitting a contact canonical frame. If $\gamma \colon \R/T\Z \to Y$ is a Reeb orbit on $(Y,\lambda)$ such that $K \circ \gamma \geq 1$ then
    $$T = \mathcal{A}(\gamma) \leq 2\pi \left\lfloor \frac{CZ(\gamma)+1}{2} \right\rfloor.$$
In particular, if $K\circ \gamma > 0$, then $CZ(\gamma) > 0$. On the other hand, if $K \circ \gamma \leq 1$, then
    $$T = \mathcal{A}(\gamma) \geq 2\pi \left\lfloor \frac{CZ(\gamma)}{2} \right\rfloor.$$
\end{theorem}

Note that the first part of Theorem \ref{thm:estimate} ensures that the inequality
\begin{equation}\label{eq: ineq_CZ}
T = \mathcal{A}(\gamma) \leq \frac{2\pi}{\sqrt{\min K\circ \gamma}}
    \left\lfloor \frac{CZ(\gamma)+1}{2} \right\rfloor,
\end{equation}
holds whenever $\gamma \colon \R/T\Z \to Y$ is a Reeb orbit on $(Y,\lambda)$ such that $K \circ \gamma > 0$. Indeed, under this assumption, one can always scale the contact form by $\sqrt{\min K \circ \gamma}$, cf. Remark \ref{rmk: constantK}. An immediate consequence of Theorem \ref{thm:estimate} concerns the type of an extremal Reeb orbit depending on the parity of its Conley--Zehnder index.
Recall that a Reeb orbit $\gamma \colon \R/T\Z \to Y$ is called \emph{elliptic} or \emph{hyperbolic} if the eigenvalues of the linearized Poincaré map 
\[ P_\gamma \colon \xi_{\gamma(0)} \to \xi_{\gamma(0)}, \]
lie on the unit circle or are non-unit real numbers, respectively. 
It follows that if $\gamma^k$ denotes the $k$-th iterate of a hyperbolic Reeb orbit $\gamma$, then 
\[ CZ(\gamma^k) = k\,CZ(\gamma). \]

\begin{corollary}
Let $(Y,\lambda)$ be a closed contact $3$-manifold admitting a contact canonical frame. 
Suppose that $\gamma \colon \R/T\Z \to Y$ is a Reeb orbit on $(Y,\lambda)$ such that $K \circ \gamma > 0$ and that the equality
\[
T = \frac{2\pi}{\sqrt{\min K\circ \gamma}}
    \left\lfloor \frac{CZ(\gamma)+1}{2} \right\rfloor,
\]
holds. Then $\gamma$ is an elliptic Reeb orbit if $CZ(\gamma)$ is odd, and (positive) hyperbolic otherwise. 

\begin{proof}
It is simple to check that $CZ(\gamma)$ is even if and only if $\gamma$ is positive hyperbolic, that is, the eigenvalues of $P_\gamma$ are positive real numbers $\lambda, \lambda^{-1}$ with $\lambda \neq 1$. 
Suppose that $CZ(\gamma)$ is an odd integer. If $\gamma$ were (negative) hyperbolic, we would have $CZ(\gamma^2) = 2\,CZ(\gamma)$ and
\[
\frac{2\cdot2\pi}{\sqrt{\min K\circ \gamma}} 
\left\lfloor \frac{CZ(\gamma)+1}{2} \right\rfloor 
= 2T = \mathcal{A}(\gamma^2) 
\leq 
\frac{2\pi}{\sqrt{\min K\circ \gamma}}
\left\lfloor \frac{CZ(\gamma^2)+1}{2} \right\rfloor,
\]
where the upper bound follows from \eqref{eq: ineq_CZ}. 
In particular, we obtain
\[
2\left(\frac{CZ(\gamma)+1}{2}\right)
= 2\left\lfloor \frac{CZ(\gamma)+1}{2} \right\rfloor
\leq
\left\lfloor \frac{2CZ(\gamma)+1}{2} \right\rfloor
= CZ(\gamma),
\]
which is a contradiction since $CZ(\gamma)$ is odd.
\end{proof}
\end{corollary}

We note that to obtain the estimates on the action of a Reeb orbit in Theorem \ref{thm:sturmcompar} or \ref{thm:estimate}, we only need the existence of the contact canonical frame in a neighborhood of the Reeb orbit.

As a second consequence of Theorem \ref{thm:estimate}, we recover a famous estimate for a systole in the case of positively curved metrics in the sphere $S^2$ due to Toponogov, see \cite{toponogov1959evaluation}. Given a Riemannian metric $g$ on $S^2$, we denote the minimum of its Gaussian curvature by $K_{min}$ and the length of the shortest closed geodesic by $L_{min}$.

\begin{corollary}\label{thm:toponogov}
    Let $g$ be a positively curved Riemannian metric on $S^2$. Then
    $$L_{min} \leq \frac{2\pi}{\sqrt{K_{min}}}.$$
    \begin{proof}
        The unit tangent bundle $S_gS^2 = \{(p,v) \in TS^2\mid g_p(v,v) = 1\}$ admits the \emph{Hilbert contact form} defined by
        $$(\lambda_g)_{(p,v)}(w) = g_{p}(v, d\pi_{(p,v)}\cdot w),$$
        where $(p,v) \in S_gS^2$, $w \in T_{(p,v)}S_gS^2$ and $\pi \colon TS^2 \to S^2$ is the canonical projection. The Reeb vector field $R$ for $\lambda_g$ coincides with the geodesic vector field defined on $S_gS^2$. In particular, the Reeb trajectories are of the form $\gamma(t) = (c(t),\dot{c}(t))$, where $c \colon I \to S^2$ is a geodesic parametrized by the arc length. Moreover, the action $\mathcal{A}(\gamma) = \int_\gamma \lambda_g$ of the Reeb trajectory $\gamma$ coincides with the length $L_g(c) = \int_I \sqrt{g_{c(t)}(\dot{c}(t),\dot{c}(t))} \ dt$ of the geodesic $c$. These facts are well known and can be found for instance in \cite{hryniewicz2013introduccao}. Further, the Conley--Zehnder index of $\gamma$ with respect to a global trivialization of $\xi$ agrees with the Morse index of the geodesic $c$, see \cite{liu2005relation}.
        
        Now it follows from the Cartan's structure equations that $(S_gS^2,\lambda_g)$ admits a contact canonical frame where $K$ agrees with the pullback of the Gaussian curvature of $g$ via $\pi$, see e.g. \cite[Section 3.1]{harris2008dynamically}. By hypothesis, we have $K>0$ and scaling the metric $g$ defining $\widetilde{g} = K_{min} g$, we obtain that the new curvature $\widetilde{K}$ satisfies $\widetilde{K} \geq 1$ everywhere. In particular, given a closed geodesic $c \colon I \to S^2$ for $\widetilde{g}$, Theorem $\ref{thm:estimate}$ yields
        \begin{equation}\label{eq:estimategeod}
            L_{\widetilde{g}}(c) = \mathcal{A}((c,\dot{c})) \leq 2\pi \left\lfloor \frac{CZ((c,\dot{c}))+1}{2} \right\rfloor.
        \end{equation}
        We recall that the shortest closed geodesic for positively curved metrics is simple and has Morse index equal to $1$, see \cite{calabi1992simple} and \cite[Theorem 4.2]{ballmann1983some}. Therefore, \eqref{eq:estimategeod} yields $L_{\widetilde{g}}(c) \leq 2\pi$ for a shortest closed geodesic $c$. Since $\widetilde{g} = K_{min}g$, we obtain the desired inequality.
        \end{proof}
\end{corollary}

\noindent{\bf Structure of the paper:} We begin Section \ref{sec: properties} providing several examples of contact canonical frames on closed contact $3$-manifolds, after that we prove some properties that the functions $\{I,J,K\}$ on the contact canonical frame must satisfy, as well as compute the full curvature tensor for the contact metric structure in terms of these three functions. In Section \ref{sec: classification_proofs} we give the proofs of the classification results in Theorems \ref{thm: K=0}, \ref{thm: K=I=0}, \ref{thm: K=1} and \ref{thm: K=-1}. In Section \ref{sec: applications} we use Sturm comparison theorem to prove Theorem \ref{thm:sturmcompar}, define the Conley--Zehnder index in the present context and prove Theorem \ref{thm:estimate} on the estimate of the action of a Reeb orbit in terms of its index and the function $K$.
\newline

\noindent\textbf{Acknowledgments:}
The authors would like to thank Lucas Ambrozio for interesting comments on the results of this manuscript. M. M. is supported by the São Paulo Research Foundation (FAPESP) under grant No. 2025/21147-7. The third author is supported by the ISF Grant No. 2445/20.

\section{Examples and properties of contact canonical frames}\label{sec: properties}

\subsection{Examples}\label{sec: examples}

In this section we give explicit contact canonical frames in some classical contact 3-manifolds: $S^3$ with the standard contact form, $\mathbb{T}^3$ with the contact form induced by the flat metric in $\mathbb{T}^2$, the unit tangent bundle of the Katok Finsler metric on $S^2$ and the boundary of the symplectic ellipsoid $E(a,b)$.

\begin{example}\label{examp: sphere}
    In $(S^3,\lambda_{std})$, let $\mathfrak{n}$ denote the unit normal vector field in $S^3\subset\mathbb{R}^4$. Here $\lambda_{std}(x) = \langle \mathbf{i}x, \cdot\rangle$ denotes the standard contact form on $S^3 \subset \C^2$. Straightforward computations show that the Reeb vector field $R_0$ is given by $R_0=\mathbf{i}\mathfrak{n}$ and the contact structure is given by $\xi:=\ker\lambda_{std}=\langle\mathbf{j}\mathfrak{n},\mathbf{k}\mathfrak{n}\rangle$, where $\mathbf{i},\mathbf{j},\mathbf{k}$ denotes the Hamilton quaternions. Furthermore, it can be shown that taking $X_1=\mathbf{j}\mathfrak{n}$ and $X_2=\frac{1}{2}\mathbf{k}\mathfrak{n}$, the relations in \eqref{eq: frame_def} are satisfied for $I$ and $J$ constant functions equal to zero and $K$ constant equal to 4. 
\end{example}

\begin{example}\label{ex: frame_T^3}
    In $\mathbb{T}^3$, take coordinates $(\theta,x,y)$ and consider the contact form 
    $$\lambda=\cos\theta dx+\sin\theta dy$$
The Reeb vector field $R$ of $\lambda$ is given by
$$R=\cos\theta\partial_x+\sin\theta\partial_y.$$
Define the frame for $\ker \lambda$ given by
$$X_1=-\sin\theta\partial_x+\cos\theta\partial_y, \quad X_2=\partial_{\theta}.$$
Notice that 
\begin{equation*}
    \left[ X_2,R \right]=X_1,\quad \left[X_1,X_2\right]=R,\quad \left[R,X_1\right]=0.
\end{equation*}
So $\{R,X_1,X_2\}$ forms a contact canonical frame for $(\mathbb{T}^3,\lambda)$, such that $I=J=K=0$ in the notation of \eqref{eq: frame_def}. Furthermore, we can put a flat Riemannian metric on $\mathbb{T}^3$ given by 
    $$g=d\theta^2+dx^2+dy^2,$$
satisfying that $\lambda=g(R,\cdot)$. In addition, notice that, for the metric $\hat{g}=dx^2+dy^2$ on $\mathbb{T}^2$, we have that $(\mathbb{T}^3,\lambda)=(S_{\hat{g}}\mathbb{T}^2,\lambda_g)$, for $\lambda_g$ the Hilbert contact form in $S_{\hat{g}}\mathbb{T}^2$ and the metric $g$ is the only lift of the metric $\hat{g}$ to $S_{\hat{g}}\mathbb{T}^2$ making $R$ orthonormal to $\xi$.
\end{example}

Now we describe the contact canonical frame on the well-known Katok Finsler metrics on $S^2$.

\begin{example}\label{ex:framekatok}
Consider the Hamiltonian
\begin{align*}
    H_a \colon T^*S^2 &\to \R \\ \nu &\mapsto H_a(\nu) = H_0(\nu) + a\langle \nu, \partial_\theta\rangle,
\end{align*}
where $a$ is a real constant, $H_0(\nu) = \Vert \nu \Vert_{g_0^*}$ is the dual norm with respect to the round metric $g_0$ and $\partial_\theta$ is the infinitesimal generator of rotations around the $z$-axis on $S^2 \subset \R^3$. When $\vert a \vert <1$, the composition $F_a = H_a \circ \mathcal{L}_{\frac{1}{2}H_a^2}$ is a Finsler metric, where $\mathcal{L}_{\frac{1}{2}H_a^2} \colon T^*S^2 \to TS^2$ denotes the Legendre transformation associated to $\frac{1}{2}H_a^2$. This is often called the Katok Finsler metric on the sphere $S^2$; see e.g. \cite{ziller1983geometry}. In particular, the level sets of $H_a$ are fiberwise convex and hence the tautological one form $\lambda = pdq$ restricts to a contact form on each of them. Consider $Y = H_a^{-1}(1)$ and denote again by $\lambda$ the latter mentioned contact form. In this level set, the Reeb vector field $R$ coincides with the Hamiltonian vector field $X_{H_a}$ defined by the equation
\begin{equation}\label{eq:Hamvfkatok}
    \omega(X_{H_a},\cdot) = dH_a,
\end{equation}
where $\omega = d\lambda = dp \wedge dq$ is the canonical symplectic form on $T^*S^2$. Given the geodesic polar coordinates $(r,\theta) \in S^2 \subset \R^3$ with the relations
$$x = \sin r \cos \theta, \quad y = \sin r \sin \theta, \quad z=\cos r,$$
let $(r,\theta,p_r,p_\theta) \in T^*S^2$ be the induced cotangent coordinates. In these coordinates, we have
$$H_a(r,\theta,p_r,p_\theta) = \sqrt{p_r^2 + \frac{p_\theta^2}{\sin^2 r}} + a p_\theta,$$
$\lambda = p_r dr + p_\theta d\theta$ and $\omega = dp_r \wedge dr + dp_\theta \wedge d\theta$. In particular, one computes the Hamiltonian vector field using the equation \eqref{eq:Hamvfkatok}:
$$R = \frac{p_r}{1-ap_\theta} \partial_r + \left(\frac{p_\theta}{(1-ap_\theta)\sin^2 r}+a\right)\partial_\theta + \left(\frac{p_\theta^2\cos r}{(1-ap_\theta)\sin^3r}\right) \partial_{p_r}.$$
We introduce the fiber angle $\psi$ such that
$$p_r = H_0\cos \psi, \quad p_\theta = H_0 \sin r \sin \psi,$$
which reduces to
$$p_r = \frac{1}{1+a\sin r \sin \psi}\cos \psi, \quad p_\theta = \frac{1}{1+a\sin r \sin \psi}\sin r \sin \psi,$$
when restricting to the level set $Y$. In coordinates $(r,\theta,\psi) \in Y$, we have
$$R = \cos \psi \ \partial_r + \left(\frac{\sin \psi}{\sin r} + a \right) \partial_\theta - \cot r \sin \psi \ \partial_\psi.$$
We note that $R = S + a\partial_\theta$, where $S$ is the (co)geodesic spray with respect to the round metric. 
From now on, we denote $\rho = 1+a\sin r \sin \psi$. Define $X_2 = \sqrt{\rho} \ \partial_\psi$. Using $S(\rho) = 0$ and $\partial_\theta(\rho) = 0$, we compute
\begin{align*}
    [X_2, R] &= [\sqrt{\rho} \ \partial_\psi, S + a\partial_\theta] \\ &= - (S+a\partial_\theta)(\sqrt{\rho})\partial_\psi - \sqrt{\rho}[S+a\partial_\theta,\partial_\psi] \\ &= -\sqrt{\rho}[S, \partial_\psi] - a\sqrt{\rho}[\partial_\theta, \partial_\psi] \\ &= \sqrt{\rho} [\partial_\psi,S] \\ &= \sqrt{\rho}H,
\end{align*}
where $H = (-\sin \psi \ \partial_r + \frac{\cos \psi}{\sin r}\ \partial_\theta - \cot r \cos \psi \ \partial_\psi)$ is the horizontal vector field with respect to the round metric. In particular, defining $X_1 = \sqrt{\rho} H$, we have $[X_2,R] = X_1$. Moreover, we have
\begin{align*}
    [R,X_1] & = [S+a\partial_\theta,\sqrt{\rho}H] \\ &= S(\sqrt{\rho})H + \sqrt{\rho}[S,H]+a\partial_\theta(\sqrt{\rho})H + a\sqrt{\rho}[\partial_\theta,H] \\ &=\sqrt{\rho}[S,H] = \sqrt{\rho} \ \partial_\psi \\ &= X_2.
\end{align*}
Finally, we compute
\begin{align*}
    [X_1,X_2] &= [\sqrt{\rho}H,\sqrt{\rho} \ \partial_\psi ] \\ &= \sqrt{\rho}H(\sqrt{\rho}) \ \partial_\psi + \sqrt{\rho}\left( - \partial_\psi(\sqrt{\rho})H + \sqrt{\rho} [H,\partial_\psi]\right) \\ &= \frac{1}{2}\left(H(\rho) \ \partial_\psi - \partial_\psi(\rho) H \right) + \rho S \\ &= \frac{1}{2}\left(-a \cos r \ \partial_\psi - a\sin r \cos \psi H \right) + \rho(R - a\partial_\theta),
\end{align*}
where we use $H(\rho) = -a\cos r$ and $R = S + a \partial_\theta$. Further, it is straightforward to check that
\begin{equation}\label{eq:partialtheta}
    \partial_\theta = \frac{\sin r}{\rho} \left( \sin \psi R + \frac{\cos \psi}{\sqrt{\rho}} X_1 + \frac{\cot r}{\sqrt{\rho}}X_2 \right).
    \end{equation}
    Using \eqref{eq:partialtheta} in the computation $[X_1,X_2]$, we obtain
    \begin{align*}
        [X_1,X_2] &= \frac{1}{2}\left(-a \cos r \ \partial_\psi - a\sin r \cos \psi H \right) + \rho(R - a\partial_\theta) \\ &= \frac{1}{2}\left(-a \cos r \ \partial_\psi - a\sin r \cos \psi H \right) + \rho R - a\sin r\left(\sin \psi R + \frac{\cos \psi}{\sqrt{\rho}} X_1 + \frac{\cot r}{\sqrt{\rho}}X_2\right) \\ &= R -\frac{3}{2}\frac{a \sin r \cos \psi}{\sqrt{\rho}} X_1 - \frac{3}{2}\frac{a \cos r}{\sqrt{\rho}} X_2,
    \end{align*}
    where we used again that $X_1 = \sqrt{\rho}H$ and $X_2 = \sqrt{\rho} \ \partial_\psi$. Therefore, we conclude that $Y = H_a^{-1}(1)$ admits a canonical contact frame $\{R,X_1,X_2\}$ where
    \begin{equation}\label{eq:framekatok}
        I = -\frac{3}{2}\frac{a \sin r \cos \psi}{\sqrt{1+a \sin r \sin \psi}}, \quad J= - \frac{3}{2}\frac{a \cos r}{\sqrt{1+a \sin r \sin \psi}}, \quad \text{and} \quad K = 1.
        \end{equation}
\end{example}

In the next lemma we study how contact frames behave under finite covering maps.

\begin{lemma}\label{lemma:covering}
    Let $(\widehat{Y},\hat{\lambda})$ and $(Y,\lambda)$ be three-dimensional closed contact manifolds. Suppose that $\Phi\colon \widehat{Y} \to Y$ is a smooth finite covering map such that $\Phi^*\lambda =c\hat{\lambda},$ where $c$ is a positive constant. Then, if $(Y,\lambda)$ assumes a contact canonical frame $\{R,X_1,X_2\}$ with the functions $I,J,K \colon Y \to \R$, then $\{\widehat{R},c\widetilde{X_1},\widetilde{X_2}\}$ is a contact canonical frame for $(\widehat{Y},\hat{\lambda})$ where $\widetilde{X}_1,\widetilde{X}_2$ are the unique lifts of $X_1,X_2$ via $\Phi$, respectively, and $\widehat{R}$ is the Reeb vector field on $(\widehat{Y},\hat{\lambda})$. In this case, the functions $\widehat{I}, \widehat{J}, \widehat{K} \colon \widehat{Y} \to \R$ are given by 
    \begin{equation*}
    \widehat{I} = I \circ \Phi, \quad \widehat{J} = c\cdot (J\circ \Phi) \quad \text{and} \quad \widehat{K} = c^2\cdot(K\circ \Phi).
\end{equation*}

\begin{proof}
    Let $\widetilde{R}$ be the unique lift of $R$ via $\Phi$. Since we have
    \begin{align*}
    c\hat{\lambda}(\widetilde{R}) &= \Phi^*\lambda(\widetilde{R}) = \lambda(R\circ \Phi) = 1, \ \text{and} \\
     c\cdot d\hat{\lambda}(\widetilde{R},\cdot) &= \Phi^*d\lambda(\widetilde{R},\cdot) = d\lambda(R\circ \Phi,d\Phi\ \cdot) = 0,
    \end{align*}
    the relation $\widetilde{R} = (1/c) \widehat{R}$ holds. The Lemma follows from this relation, the uniqueness of the lift of a vector field on $Y$ to a vector field on $\widehat{Y}$ via $\Phi$ and the relation
    $$d\Phi([V_1,V_2]) = [d\Phi(V_1), d\Phi(V_2)],$$
    for any vector fields $V_1,V_2$ on $\widehat{Y}$. In fact, we have the following
    \begin{align*}
        \left[\widetilde{X}_2,\widehat{R} \right] &= \left[\widetilde{X}_2,c\widetilde{R} \right] = c\widetilde{X}_1, \\  \left[c\widetilde{X}_1,\widetilde{X}_2 \right] &= c\widetilde{R} + (I \circ \Phi) c\widetilde{X}_1 + c\cdot (J\circ \Phi) \widetilde{X}_2, \\ \left[\widehat{R},c\widetilde{X}_1\right] &= \left[c\widetilde{R},c\widetilde{X}_1\right] =c^2 \cdot (K\circ \Phi)\widetilde{X}_2.
        \end{align*}
\end{proof}
\end{lemma}

We now use Example \ref{ex:framekatok} and Lemma \ref{lemma:covering} to prove the existence of contact canonical frames in symplectic ellipsoids, see the definition in \eqref{eq:defellipsoid}.

\begin{example}\label{ex: frame_ellip}
    Given any Finsler metric $F$ on $S^2$, there exists a double covering $\Phi \colon S^3 \to S_FS^2$ such that\footnote{We recall that $\lambda_{std} = 2\lambda_0$ on $S^3$.} $\Phi^*\lambda_F = 4h \lambda_0$, where $\lambda_F$ is the Hilbert contact form on $S_FS^2$ with respect to $F$, $\lambda_0$ is the restriction to $S^3$ of the standard Liouville form on $\C^2$ and $h \colon S^3 \to \R_{>0}$ is a positive smooth function; see \cite[Section 4]{harris2008dynamically}. When $F$ is the Katok Finsler metric $F_a = H_a \circ \mathcal{L}_{\frac{1}{2}H_a^2}$ discussed in Example \ref{ex:framekatok}, $h$ is the contact form corresponding to the ellipsoid
    \[\partial E\left(\frac{\pi}{1+a},\frac{\pi}{1-a}\right) = \left\{(z_1,z_2) \in \C^2 \mid \pi\left(\frac{1+a}{\pi}\vert z_1 \vert^2 + \frac{1-a}{\pi}\vert z_2 \vert^2\right) = 1\right\},\]
    i.e. $h$ is such that there exists a diffeomorphism $g \colon S^3 \to \partial E\left(\frac{\pi}{1+a},\frac{\pi}{1-a}\right)$ such that $g^*\lambda_0 = h\lambda_0$. In this case, we have the double covering
    $$\Phi \circ g^{-1} \colon \partial E\left(\frac{\pi}{1+a},\frac{\pi}{1-a}\right) \to S_{F_a}S^2,$$
    such that $(\Phi \circ g^{-1})^*\lambda_{F_a} = 4\lambda_0$. Therefore, by Example \ref{ex:framekatok} and Lemma \ref{lemma:covering}, we conclude that the complex ellipsoid $\partial E\left(\frac{\pi}{1+a},\frac{\pi}{1-a}\right)$ equipped with the standard Liouville form admits a contact canonical frame such that
    $$I = I_a \circ \Phi \circ g^{-1}, \quad J = 4\cdot (J_a \circ \Phi \circ g^{-1}) \quad \text{and} \quad K = 16, $$
    where $I_a$ and $J_a$ are the corresponding functions in the Katok example given by \eqref{eq:framekatok}. We note that when $a=0$, this recovers Example \ref{examp: sphere}.
        Furthermore, given a complex ellipsoid $E(a,b)$ as in \eqref{eq:defellipsoid}, it is simple to check that
   $$E(a,b) = \sqrt{\frac{2ab}{\pi(a+b)}} \cdot E\left(\frac{\pi}{1+c},\frac{\pi}{1-c}\right),$$
    where $c = (b-a)/(a+b)$. In particular, given the behavior of the canonical contact frame under scaling by constant, cf. Remark \ref{rmk: constantK}, and the computation for $\partial E\left(\frac{\pi}{1+a},\frac{\pi}{1-a}\right)$ above, we conclude that $\partial E(a,b)$ admits a contact canonical frame for the contact form $\lambda_0$ such that
    \begin{equation}\label{eq:Kellipsoid}
    K = 16 \cdot \left(\frac{\pi(a+b)}{2ab}\right)^2.
    \end{equation}
\end{example}

\begin{remark}
    Even though we are only considering contact canonical frames on closed 3-manifolds, it follows from Example \ref{ex: frame_ellip} that the boundary of the symplectic cylinder $B^2(r)\times \mathbb{R}^2$ also admits a contact canonical frame with $K=16\left(\frac{\pi}{r}\right)^2$. Furthermore, it also follows from Example \ref{ex: frame_ellip} that the boundary of the polydisk 
    $$P(a,b):=\{(z_1,z_2)\in \mathbb{C}^2| \|z_1\|^2\leq a,\|z_2\|^2\leq b\},$$
    admits a canonical frame defined in the complement of the set where $\|z_1\|^2= a,\|z_2\|^2= b$.
\end{remark}


\subsection{Properties}

We begin this section by studying the relations that the Jacobi identity forces upon the functions $\{I,J,K\}$, thanks to the structure equations in \eqref{eq: frame_def}. This is the content of the next proposition.

\begin{proposition}\label{eq: prop: Jacobi}
    Let $\lambda$ be a contact form in $Y$ and $R$ its Reeb vector field. Let $\{R,X_1,X_2\}$ be a \textit{contact canonical frame} for $\lambda$. Then 
    \begin{itemize}
        \item $R(I)=J$,
        \item $IK+R(J)+X_2(K)=0$.
    \end{itemize}

\end{proposition}
\begin{proof}
    Using the Jacobi identity
    $$[R,[X_1,X_2]]+[X_1,[X_2,R]]+[X_2,[R,X_1]]=0,$$
    together with \eqref{eq: frame_def} we have that
    \begin{equation}\label{eq: Jacobi}
       \left(R(I)-J\right)X_1+\left(IK+R(J)+X_2(K)\right)X_2=0, 
    \end{equation}
and since $X_1$ and $X_2$ must be linearly independent, the claims follow.
\end{proof}

As a first consequence, we prove Theorem \ref{thm:landsbergctc}.

\begin{proof}[Proof of Theorem \ref{thm:landsbergctc}]
    Since $J \equiv 0$, Proposition \ref{eq: prop: Jacobi} yields
    $$IK + X_2(K) = 0.$$
    Therefore, if $\psi(t)$ is a flow line of $X_2$, we have
    $$\dot{K}(t) + I(t)K(t) = 0,$$
    where the dependence in parameter $t$ is understood as the composition with $\psi$. Solving this ODE leads to equation \eqref{eq:curvatureode}.
\end{proof}

The relations in Proposition \ref{eq: prop: Jacobi} will be directly used in the proofs of Theorems \ref{thm: K=0} and \ref{thm: K=-1}, as well as in the following computations of the curvature tensor for the Levi-Civita connection of the Riemannian metric defined in \eqref{eq: riem_metric}. This result in turn, will be useful in the proof of Theorem \ref{thm: K=I=0}.

\begin{proposition}\label{prop: curvature_tensor}
    The Riemannian curvature tensor on a contact 3-manifold $Y$ admitting a contact canonical frame and whose metric is defined as in \eqref{eq: riem_metric}, is given by the equations
    \begin{equation}\label{eq:curvature_components_onecolumn}
\begin{aligned}
\mathcal{R}(R,X_1)R &= \big(\tfrac{3}{4}K^2-K\big)X_1 - \tfrac{1}{2}R(K)\,X_2,\\
\mathcal{R}(R,X_1)X_1 &= \big(-\tfrac{3}{4}K^2+K\big)R + \big(-J-\tfrac{1}{2}X_1(K)+KJ\big)X_2,\\
\mathcal{R}(R,X_1)X_2 &= \tfrac{1}{2}R(K)\,R + \big(J+\tfrac{1}{2}X_1(K)-KJ\big)X_1,\\[4pt]
\mathcal{R}(R,X_2)R &= -\tfrac{1}{2}R(K)\,X_1 - \tfrac{1}{4}K^2 X_2,\\
\mathcal{R}(R,X_2)X_1 &= \tfrac{1}{2}R(K)\,R + \big(I(K-1)+\tfrac{1}{2}X_2(K)\big)X_2,\\
\mathcal{R}(R,X_2)X_2 &= \tfrac{1}{4}K^2 R + \big(-I(K-1)-\tfrac{1}{2}X_2(K)\big)X_1,\\[4pt]
\mathcal{R}(X_1,X_2)R &= \big(-\tfrac{1}{2}X_1(K)-J+KJ\big)X_1 + \big(-I+IK+\tfrac{1}{2}X_2(K)\big)X_2,\\
\mathcal{R}(X_1,X_2)X_1 &= \big(\tfrac{1}{2}X_1(K)-JK+J\big)R + \big(-\tfrac{1}{4}K^2 + X_2(I)-X_1(J)+I^2+J^2\big)X_2,\\
\mathcal{R}(X_1,X_2)X_2 &= \big(X_1(J)-X_2(I)+\tfrac{1}{4}K^2 - I^2 - J^2\big)X_1 + \big(-I(K-1)-\tfrac{1}{2}X_2(K)\big)R.
\end{aligned}
\end{equation}
In particular, if $I=J=K\equiv 0$, we get that the curvature tensor is zero and $Y$ is flat.
\end{proposition}
\begin{proof}
Using the structure equations in \eqref{eq: frame_def}, the Koszul formula and the torsion-free property of the Levi-Civita connection, we get that
\begin{equation}\label{eq:derivatives_frame}
\setlength{\tabcolsep}{2.5em} 
\renewcommand{\arraystretch}{1.2} 
\begin{array}{rclcrclcrcl}
\nabla_{R}R &=& 0, &&
\nabla_{R}X_1 &=& \tfrac{1}{2}K X_2, &&
\nabla_{R}X_2 &=& -\tfrac{1}{2}K X_1,\\[3pt]
\nabla_{X_1}R &=& -\tfrac{1}{2}K X_2, &&
\nabla_{X_1}X_1 &=& -I X_2, &&
\nabla_{X_1}X_2 &=& I X_1 + \tfrac{1}{2}K R,\\[3pt]
\nabla_{X_2}R &=& \big(1-\tfrac{1}{2}K\big) X_1, &&
\nabla_{X_2}X_1 &=& \big(\tfrac{1}{2}K - 1\big)R - J X_2, &&
\nabla_{X_2}X_2 &=& J X_1.
\end{array}
\end{equation}
With these computations at hand and using Proposition \ref{eq: prop: Jacobi} for further simplifications, we can compute the Riemannian curvature tensor to get the relations in \eqref{eq:curvature_components_onecolumn}.
\end{proof}

The following property is taken from the discussion in \cite[Section 3]{harris2008dynamically}. Since it shall be useful for this work, we present its proof here.

\begin{lemma}\label{lemma:ode}
    Let $(Y,\lambda)$ be a closed contact $3$-manifold and $\{X_1,X_2\}$ be a contact canonical frame for $\lambda$. Denote by $\eta = x X_1 + yX_2$. Then for the Reeb flow $\varphi_t$,
    $$d\varphi_t(\eta) = x(t) X_1 + y(t) X_2,$$
    where $x(t)$ and $y(t)$ solve the initial value problem 
    $$\dot{x}(t) = y(t),\quad \dot{y}(t) = -Kx(t),$$
    with $x(0) = x$ and $y(0) = y$.
    \begin{proof}
Writing $d\varphi_t(\eta) = x(t) X_1 + y(t) X_2$, we obtain
$$\eta=x(t)d\phi_{-t}(X_1)+y(t)d\phi_{-t}(X_2).$$
    Differentiating with respect to $t$ we obtain
    $$0 = \overset{.}{x}X_1 + x[R, X_1] + \overset{.}{y}X_2 + y[R, X_2].$$
    Using the equations defining the frame in \eqref{eq: frame_def} and regrouping, we have:
    $$0 = (\overset{.}{x}-y)X_1 + (\overset{.}{y}+ xK)X_2,$$
    which gives the equations
   $$\begin{aligned}
   \overset{.}{x}&=y,\\
\overset{.}{y}&=-Kx.
       \end{aligned}$$
    \end{proof}
\end{lemma}

\section{On contact canonical frames with constant $K$}\label{sec: classification_proofs}
\subsection{The case $K\equiv 0$}

In this section we give the proofs of Theorems \ref{thm: K=0} and \ref{thm: K=I=0}.

\begin{proof}[Proof of Theorem \ref{thm: K=0}]
    We begin with the proof of item \ref{it: I=J=0}. From $K\equiv 0$ and Proposition \ref{eq: prop: Jacobi}, we have that $R(J)=0$ and $R(I)=J$. The former implies that $J$ is constant along flow lines of the Reeb vector field. Assume there is a flow line in which this constant is non-zero. Then, restricted to a flow line of the Reeb vector field, $I$ is linear in the parameter $t$ of the flow line. So: 
    \begin{itemize}
        \item if the flow line is not periodic, then we have that the smooth function $I$ is unbounded in the compact manifold $Y$, which is a contradiction.
        \item if the flow line is periodic then we must have $J=0$ along the flow line and $I$ constant. 
    \end{itemize}  
So $J\equiv 0$, and $I$ is constant along flow lines.

We first claim that $X_1(I)=0$. To see this, consider the vector field $X_I$ defined as the unique vector field in $\ker \lambda$ such that $\iota_{X_I}d\lambda=dI$, restricted to $\ker \lambda$. Make $X_I=\alpha X_1+\beta X_2$ and $d\varphi_t(X_I)=\alpha(t) X_1+\beta(t) X_2$, for $\varphi_t$ the flow of the Reeb vector field $R$. It follows from Lemma \ref{lemma:ode} by plugging in $K=0$, that $\alpha(t),\beta(t)$ satisfies the system of ordinary differential equations
$$\overset{.}{\alpha}(t)=\beta(t),\quad \overset{.}{\beta}(t)=0,$$
which means that $\alpha(t)=C_1t+C_2$ and $\beta(t)=C_1$, for real constants $C_1$ and $C_2$. Notice that by changing the initial conditions of this differential equation, we can think of $\alpha$ as a continuous global function on $Y$, that, as we just concluded, is linear along flow lines of $R$. Therefore, along a non-closed Reeb orbit, $\alpha$ is unbounded, contradicting the compactness of $Y$, so $C_1=0$. Whereas, along a closed Reeb orbit, $\alpha$ clearly cannot have $C_1\neq 0$. So, $\beta$ is globally equal to zero and hence by the definition of $X_I$, we have that $X_1(I)=0$ as claimed. 

Now, since by hypothesis we have that $[R,X_1]=0$ and both $R$ and $X_1$ are non-vanishing, it follows from Frobenius integrability Theorem, that the sub-bundle $\langle R,X_1\rangle\subset TY$ gives rise to a regular foliation of $Y$ by surfaces. Furthermore, since $X_1(I)=R(I)=0$, we must have that any leaf of this foliation is contained in a level set of $I$ and, any connected component of a level set of $I$ must be contained on a leaf of the foliation. Therefore, connected components of a level set are regular leaves of the foliation. Since level sets are closed submanifolds of the compact manifold $Y$, their connected components must be compact.  Any such connected component is a leaf of the foliation and in any such leaf there is a $\mathbb{R}^2$-action induced by the flows of $R$ and $X_1$. This action must have discrete isotropy group (since the leaf is a surface) and this subgroup must be necessarily isomorphic to $\mathbb{Z}^2$. Hence, the leaf is topologically a $T^2$, and the $\mathbb{R}^2$-action induces a free and proper $T^2$-action on the leaf. This adds up to a global proper and free $T^2$-action on $Y$. The quotient of $Y$ by this $T^2$-action must be a compact 1-dimensional manifold, hence topologically $S^1$, proving \ref{it: T^2-bundle}. 

Notice now that $I$ induces a smooth function $\tilde{I}$ on $S^1$. We now prove \ref{it: induced_I}. Notice that the linearization of the fibration map $Y\to S^1$ sends $X_2$ to a non-zero vector on $TS^1$. Since we intend to compute the monodromy around the base of this fibration, we will sometimes ahead use interchangeably $X_2$ or its image by its linearization, as a slight abuse of notation. Consider the adjoint action $\textup{ad}_{X_2}:\langle R,X_1\rangle \to \langle R,X_1\rangle$ given by 
$$\textup{ad}_{X_2}(W)=[X_2,W], $$
for any $W\in \langle R,X_1\rangle$. This is a linear map, easily seen to be represented by the matrix
\begin{equation*}
    M(\theta)=\begin{pmatrix}
0 & -1\\[4pt]
1 & -\tilde{I}(\theta)
\end{pmatrix}
\end{equation*}
as consequence of \ref{eq: frame_def} for $K\equiv J\equiv 0$. If $\Phi(\theta)$ denotes the linear map on the fiber of the $T^2$-fibration, induced by flowing along $X_2$ from base parameter 0 to $\theta$, then $\Phi(\theta)$ satisfies the linear ODE
$$\frac{d}{d\theta}\Phi(\theta)=M(\theta)\Phi(\theta), \quad \Phi(0)=\textup{Id}.$$
The monodromy after one loop around the base is $\Phi:=\Phi(l)$, where $l$ is the base-period\footnote{As consequence of Theorem \ref{thm: K=0}, for any point $p$ in a level surface of $I$, its flow by $X_2$ will at some time in the future intersect the level surface again. We then measure the length of this flow line with the contact form $\lambda$, and this will be the value of $l$.}. This $\Phi$ is the path-ordered exponential matrix
$$\Phi= \textup{exp}\left(\mathcal{T}\int_0^lM(\theta)d\theta\right),$$
where $\mathcal{T}$ denotes the time-ordering operator.
Differentiating $\det\Phi(\theta)$ we also get the equation
$$\frac{d}{d\theta}\det\Phi(\theta)=\textup{tr }(M(\theta))\det\Phi(\theta),$$
and solving we get
$$\det\Phi=\textup{exp}\left(\int_0^l\textup{tr }M(\theta)d\theta\right)=\textup{exp}\left(-\int_0^l\tilde{I}(\theta)d\theta\right).$$
Since the monodromy matrix $\Phi$ must have $\textup{det }\Phi=1$, we must have that $\int_{S^1}\tilde{I}(\theta)d\theta=0$, as claimed.

\end{proof}

\begin{proof}[Proof of Theorem \ref{thm: K=I=0}]
  Notice that $K\equiv 0$ implies that $J\equiv 0$ thanks to \ref{it: I=J=0} in Theorem \ref{thm: K=0}, so the flatness statement in \ref{it: flat_metric} follows from Proposition \ref{prop: curvature_tensor}, while the completeness of the metric is just due to the fact that $Y$ is compact. We now prove \ref{it: cryst_group}, notice first that the relations in \eqref{eq: frame_def} now read
\begin{equation}\label{eq: frame_reduced}
    \left[ X_2,R \right]=X_1,\quad \left[X_1,X_2 \right]=R,\quad \left[R,X_1\right]=0.
\end{equation}
These are exactly the commutation relations of the Lie algebra $\frak{se}(2)$ of orientation-preserving Euclidean motions (rotations and translations) of the plane. This Lie algebra corresponds with case VII in the list of all three-dimensional real Lie algebras up to isomorphism given in \cite{Bianchi2001}. Thus, the vector fields $R,X_1,X_2$ give $Y$ the structure of a manifold equipped with a global frame whose Lie algebra is $\frak{se}(2)$. Lie's third theorem states that every finite-dimensional real Lie algebra $\frak{g}$ is the Lie algebra of some simply-connected Lie group $G$ and such group is unique up to isomorphism, see for example \cite{SerreLie}. The simply-connected Lie group integrating $\frak{se}(2)$ is the universal cover $\widetilde{SE}(2)$ of the Lie group $SE(2)$ of orientation-preserving Euclidean motions of the plane given as
\[
SE(2)
= \mathbb{R}^2 \rtimes S^1
= \{\,((x,y),t)\mid (x,y)\in\mathbb{R}^2,\ t\in S^1\,\}.
\]
Thus, $\widetilde{SE}(2)$ as a Lie group that  topologically looks like $\mathbb{R}^3$ and whose group operation is given by the semi-direct product $\mathbb{R}^2\rtimes \mathbb{R}$, defined as follows: 
$$(x_1,y_1,t_1)\cdot(x_2,y_2,t_2)=((x_1,y_1)+R_{t_1}(x_2,y_2),t_1+t_2),$$
where $R_{t}$ denotes the rotation in the plane by angle $t\in \mathbb{R}$.

The relations in \eqref{eq: frame_reduced}, then imply that, locally, $Y$ looks like the simply-connected Lie group $\widetilde{SE}(2)\cong\mathbb{R}^3$. Since, $Y$ is compact, the only way this can happen is if $Y\cong \widetilde{SE}(2)/\Gamma$, where $\Gamma\subset \widetilde{SE}(2)$ is a discrete subgroup acting freely and properly discontinuously, such that the quotient is compact. Such a 
$\Gamma$ will be a lattice in $\widetilde{SE}(2)$. It follows then from \ref{it: flat_metric}, that $Y$ is locally modeled with the geometry of the Euclidean space $\mathbb{E}^3$.

By item \ref{it: flat_metric} the Riemannian metric $\langle\cdot,\cdot\rangle$ on $Y$ defined in \eqref{eq: riem_metric} is complete and flat. Let $(\widetilde{SE}(2), \langle\langle\cdot,\cdot\rangle\rangle)$ denote the Riemannian universal cover of $(Y,\langle\cdot,\cdot\rangle)$ equipped with the lifted metric $ \langle\langle\cdot,\cdot\rangle\rangle$. Then $(\widetilde {SE}(2), \langle\langle\cdot,\cdot\rangle\rangle)$ is complete, simply connected and has vanishing Riemann curvature tensor.

It is a consequence of the Cartan-Ambrose-Hicks Theorem on spaces of constant sectional curvature, that any complete, simply connected, flat Riemannian $n$-manifold is isometric to Euclidean space $(\mathbb{R}^n,g)$, where $g$ is the usual Euclidean metric, see for example \cite{doCarmoRiem}. The deck transformation group $\textup{Deck}(\widetilde {SE}(2)/Y)\cong\pi_1(Y)$ acts on $(\widetilde {SE}(2), \langle\langle\cdot,\cdot\rangle\rangle)\cong(\mathbb{R}^3,g)$ by orientation preserving isometries. Thus, under the identification $\widetilde{SE}(2)\cong\mathbb R^3$ the group
\[
\Gamma := \textup{Deck}(\widetilde{SE}(2)/Y)
\]
is realized as a discrete, torsion-free subgroup of the full Euclidean isometry group $\textup{Isom}(\mathbb R^3)\cong\mathbb R^3\rtimes O(3)$ acting freely and cocompactly on $\mathbb R^3$. By definition such a subgroup is a \emph{torsion-free crystallographic group}, and in this case, a Bieberbach subgroup of $\textup{Isom}(\mathbb{R}^3)$, see \cite{Charlap}.

Therefore
$$
Y \cong \widetilde{SE}(2)/\Gamma  \cong \mathbb R^3/\Gamma,
$$
with $\Gamma$ a torsion-free crystallographic group, proving item \ref{it: cryst_group}.

We now prove \ref{it: quot_of_R^3}. Let \(T\cong\mathbb R^3\) denote the subgroup of translations of \(\textup{Isom}(\mathbb R^3)\). Set
\[
\Lambda := \Gamma \cap T.
\]
By the first and second Bieberbach Theorems, see \cite{Charlap}, the following hold for any crystallographic group:
\begin{itemize}
  \item \(\Lambda\) is a full rank lattice in \(T\), hence \(\Lambda\cong\mathbb Z^3\).
  \item The quotient \(\Gamma/\Lambda\) is finite (this quotient is the holonomy group).
\end{itemize}
In particular \(\Lambda\) is a subgroup of finite index in \(\Gamma\). Passing to quotients, the covering \(\mathbb R^3/\Lambda \to \mathbb R^3/\Gamma\) has deck group \(\Gamma/\Lambda\), which is finite. But \(\mathbb R^3/\Lambda\) is the standard 3-torus \(T^3\). Thus \(T^3\) is a finite (regular) cover of \(Y\):
\[
T^3 \cong \mathbb R^3/\Lambda \longrightarrow \mathbb R^3/\Gamma \cong Y,
\]
and item \ref{it: quot_of_R^3} follows.

Finally, using \ref{it: cryst_group} we have that \ref{it: all_possibilities}  follows from the theory of Bieberbach groups and the complete list of orientable flat 3-manifolds, using the restrictions that \ref{it: T^2-bundle} in Theorem \ref{thm: K=0} and \ref{it: quot_of_R^3} impose on $Y$, see \cite{Charlap} and also \cite{Hillman2002}[Chapter 8], for complete list of Bieberbach groups in dimension 3.

\end{proof}

\subsection{The case $K\equiv 1$}

We start this section by giving the proof of Theorem \ref{thm: K=1}.

\begin{proof}[Proof of Theorem \ref{thm: K=1}]

First, we note that, by the nature of the proof and the fact that $\Omega \subset \R^2_{\geq 0}$ is a smooth star-shaped region, we may assume without loss of generality that $\mathbb{X}_{\Omega}$ is a toric domain of the form
$$
\mathbb{X}_{\Omega}=\{(z_1,z_2)\in \mathbb{C}^2\mid 0 \leq \pi \vert z_1 \vert^2 \leq a, \ \pi |z_2|^2 \leq f(\pi |z_1|^2) \},
$$
where $f:[0,a]\rightarrow [0,b]$ is a smooth function such that $f(0)=b>0, f(a)=0, f'(a)\neq 0$ and $\lim_{t\to 0^+}f'(t)\neq \infty$. Let $Y=\partial \mathbb{X}_{\Omega}$ and consider the following map
$$
   \begin{array}{cccc}
    \Psi \ : & \! [0,a]\times T^2 & \! \longrightarrow
                   & \! Y \\
    & \! (t,\theta_1,\theta_2) & \! \longmapsto
                   & \! (\sqrt{t/\pi} e^{i\theta_1},\sqrt{\frac{f(t)}{\pi}} e^{i\theta_2})
   \end{array}
   $$
   The map $\Psi$ has the following properties:
   \begin{itemize}
       \item $\Psi$ restricted to $(0,a)\times T^2$ is an embedding.
       \item $Y\setminus \mathrm{Im}(\Psi|_{(0,a)\times T^2})$ is the union of two Reeb orbits, we will refer to these as \emph{exceptional Reeb orbits}.
       \item $\Psi^*\lambda_0|_{Y}=\frac{t}{2\pi}d\theta_1+\frac{f(t)}{2\pi} d\theta_2$
   \end{itemize}
The exceptional Reeb orbits, which we denote by $\gamma_-$ and $\gamma_+$, are given by
   $$
   \gamma_-(\theta)=(0,\sqrt{b/\pi} e^{i\theta})\quad \mathrm{and}\quad \gamma_+(\theta)=(\sqrt{a/\pi} e^{i\theta},0).
   $$
 In these coordinates the Reeb vector field is given by
   $$
   R=\frac{2\pi}{f(t)-tf'(t)}(-f'(t)\partial_{\theta_1}+\partial_{
   \theta_2
   })=u(t)(-f'(t)\partial_{\theta_1}+\partial_{
   \theta_2
   }),
   $$
where $u(t):=\frac{2\pi}{f(t)-tf'(t)}$. Now, write the frame as follows
\begin{align*}
X_1 &= A_1\partial_t + A_2\partial_{\theta_1} + A_3\partial_{\theta_2},\\
X_2 &= B_1\partial_t + B_2\partial_{\theta_1} + B_3\partial_{\theta_2},
\end{align*}
for functions $A_1,A_2,A_3,B_1,B_2,B_3$ on $\partial \mathbb{X}_{\Omega}$. Using that $[R,X_1]=X_2$ we obtain
   \begin{equation}\label{eq:EDP_system_1}
\begin{aligned}
B_1 &= -u(t)\bigl(f'(t)\partial_{\theta_1}A_1 - \partial_{\theta_2}A_1\bigr),\\
B_2 &= \partial_t(u(t)f'(t))A_1 - u(t)f'(t)\partial_{\theta_1}A_2 + u(t)\partial_{\theta_2}A_2,\\
B_3 &= -u(t)f'(t)\partial_{\theta_1}A_3 - u'(t)A_1 + u(t)\partial_{\theta_2}A_3.
\end{aligned}
\end{equation}
Now using that $[X_2,R]=X_1$, we obtain
\begin{equation}\label{eq:EDP_system_2}
\begin{aligned}
A_1 &= u(t)\bigl(f'(t)\,\partial_{\theta_1}B_1 - \partial_{\theta_2}B_1\bigr),\\
A_2 &= -B_1\bigl(u'(t)f'(t) + u(t)f''(t)\bigr)
      + u(t)\bigl(f'(t)\,\partial_{\theta_1}B_2 - \partial_{\theta_2}B_2\bigr),\\
A_3 &= B_1\,u'(t)
      + u(t)\bigl(f'(t)\,\partial_{\theta_1}B_3 - \partial_{\theta_2}B_3\bigr).
\end{aligned}
\end{equation}
Note that if $A_1$ is a constant solution of these equations, then $A_1$ must be zero. But if $A_1$ is zero, then both $X_1$ and $X_2$ will be tangent to the Morse-Bott tori. Since $X_1,X_2$ and $R$ are linearly independent, we must have that $A_1$ cannot be a constant function.

It is not hard to see, putting together equations \ref{eq:EDP_system_1} and \eqref{eq:EDP_system_2}, that the functions $A_1$ and $B_1$ satisfy the following PDE
$$
     \bigl(1+u(t)^2(f'(t)\partial_{\theta_1}-\partial_{\theta_2})^2 \bigr)P=0.
$$
    Write the Fourier expansion of $A_1$ and $B_1$
    $$
    A_1=\sum_{(k_1,k_2)\in \mathbb{Z}^2}a_{(k_1,k_2)}(t)e^{i(k_1\theta_1+k_2\theta_2)}.
    $$
    and
    $$
    B_1=\sum_{(k_1,k_2)\in \mathbb{Z}^2}b_{(k_1,k_2)}(t)e^{i(k_1\theta_1+k_2\theta_2)}.
    $$
    Let $t_0$ be any number in $(0,a)$, we can not have both $A_1$ and $B_1$ equal zero at $t_0$ as vector field in $\{t_0\}\times T^2$, since the Reeb vector field is tangent to this torus and $X_1,X_2$ and $R$ are linearly independent everywhere. Therefore, there exists some $(k_1,k_2)$ pairs of integer in $\mathbb{Z}^2\setminus \{(0,0)\}$, such that 
    $$
    a_{k_1,k_2}(t_0)\neq 0\quad \mathrm{or}\quad b_{k_1,k_2}(t_0)\neq 0.
    $$

    For a fixed $t_0$ suppose that there exists a pair of integers $(k_1,k_2) \neq (0,0)$ such that $a_{k_1,k_2}(t_0)\neq 0$. By continuity $a_{k_1,k_2}(t)\neq 0$, for all $t$ belonging to a neighborhood of $t_0$. Now we observe that
    $$
    (f'\partial_{\theta_1}-\partial_{\theta_2})e^{i(k_1\theta_1+k_2\theta_2)}=i(f'(t)k_1-k_2)e^{i(k_1\theta_1+k_2\theta_2)},
    $$
    thus
    $$
    (f'\partial_{\theta_1}-\partial_{\theta_2})^2e^{i(k_1\theta_1+k_2\theta_2)}=-(f'(t)k_1-k_2)^2e^{i(k_1\theta_1+k_2\theta_2)}.
    $$
    We derive the equality
    $$
    a_{(k_1,k_2)}(t)=u(t)^2a_{(k_1,k_2)}(t)(f'(t)k_1-k_2)^2.
    $$
    Since $a_{k_1,k_2}(t)$ is non zero in the neighborhood of $t_0$ we obtain
    $$
    (f'(t)k_1-k_2)^2=\frac{1}{u(t)^2}=\frac{(f(t)-tf'(t))^2}{4\pi^2},
    $$
    so that $f$ must satisfy
    $$
    f'(t)=\frac{2\pi k_2+f(t)}{2\pi k_1+t}.
    $$
    Differentiating again, we obtain
    $$
    f''(t)=\frac{f'(t)(2\pi k_1+t)-(2\pi k_2+f(t))}{(2\pi k_1+t)^2}=\frac{(2\pi k_2+f(t))-(2\pi k_2+f(t))}{(2\pi k_1+t)^2}=0.
    $$
    Therefore, $f$ is linear in a neighborhood of $t_0$. The same analysis applies when $b_{k_1,k_2}(t_0) \neq 0$, and we conclude the following: for every $t_0 \in (0, a)$, there exists a neighborhood of $t_0$ in which $f$ is linear. Hence, $f$ must be a linear function on the entire interval $(0, a)$ and thus $X_\Omega = E(a,b)$. The relation
    $$\pi(a+b) = \frac{ab}{2}$$
    follows from the formula \eqref{eq:Kellipsoid} of $K$ for the ellipsoid $(\partial E(a,b),\lambda_0)$.
\end{proof}

\subsection{The case $K\equiv -1$}

Finally, to close this section, we give the proof of Theorem \ref{thm: K=-1}.

\begin{proof}[Proof of Theorem \ref{thm: K=-1}]
    Suppose that $\{R,X_1,X_2\}$ is a contact canonical frame for $\lambda$ such that $K\equiv -1$. In this case, Proposition \ref{eq: prop: Jacobi} gives $R(I) = J$ and $R(J) = I$, where $R$ is the Reeb vector field for $(Y,\lambda)$. Therefore, the function $I$ is the solution for the ordinary differential equation 
    $$R^2(I) = R(R(I)) = I.$$
    This implies that along a Reeb flow line $t\mapsto \varphi^t(y)$, the function $I$ is given by $$I(t) = A\cosh{t} + B\sinh{t}.$$ 
    Since $Y$ is closed and $I \colon Y \to \R$ is smooth, $I$ is bounded and we must have $A=B=0$. Therefore, $I\equiv 0$. By the symmetry of the problem in $I$ and $J$, we conclude also that $J \equiv 0$.

    Thus, the contact canonical frame satisfies the bracket relations
    \begin{equation}\label{eq:relations}
        [X_2,R] = X_1, \quad [X_1,X_2] = R, \quad [R,X_1] = -X_2.
    \end{equation}
    These are exactly the relations satisfied by the Lie Algebra $\mathfrak{sl}(2) \cong \mathfrak{so}(2,1)$. The simply connected Lie group integrating the latter is the universal covering of $SL(2,\R)$ that we denote here by $\widetilde{SL}(2,\R)$. In this case, $Y$ has a $SL_2$ Thurston geometry, i.e., it is given by the quotient $\widetilde{SL}(2,\R)/\Gamma$, where $\Gamma$ is a discrete, torsion-free, cocompact subgroup acting freely on $\widetilde{SL}(2,\R)$.
    
    The relations \eqref{eq:relations} yield a topological fibration
    $$\R \hookrightarrow \widetilde{SL}(2,\R) \to \mathbb{H}^2,$$
    where the lift of $R$ direction spans the $\R$ fiber and the directions determined by the lifts of $X_1$ and $X_2$ generate the hyperbolic plane $\mathbb{H}^2$. In this case, $\Gamma$ acts on $\widetilde{SL}(2,\R)$ by isometries and preserving the fibers. Since $Y = \widetilde{SL}(2,\R)/\Gamma$ is compact, this yields a Seifert fibration with nonzero Euler class
    \begin{equation}\label{eq:bundle}
        S^1 \hookrightarrow Y \xrightarrow{\pi} \mathbb{H}^2/\Gamma \cong X,
    \end{equation}
    where $X$ is a two-dimensional orbifold such that $\chi_{orb}(X)<0$, and the fibers are generated by the Reeb direction; see e.g. \cite{scott1983geometries} or \cite{martelli2016introduction}.
    \end{proof}

\section{Applications}\label{sec: applications}

We now give proofs of the results obtained as applications of contact canonical frames stated in Section \ref{sec: applications1}.

\subsection{Proof of Theorem \ref{thm:sturmcompar}}
We recall Sturm’s result following \cite[Chapter 3]{u1978ordinary}.

\begin{theorem}[Sturm Comparison]\label{thm: sturm}
    Let $\phi_1(t)$ and $\phi_2(t)$ be nontrivial solutions of the equations
    \begin{align}
        \ddot{y}(t) + q_1(t)\, y(t) = 0, \\
        \ddot{y}(t) + q_2(t)\, y(t) = 0,
    \end{align}
    respectively, on an interval where $q_1(t) \ge q_2(t)$ and $q_1 \ne q_2$. Then, if $t_1 < t_2$ are zeros of $\phi_2$, i.e.\ $\phi_2(t_1) = \phi_2(t_2) = 0$, there exists a zero $z$ of $\phi_1$ in $(t_1, t_2)$.
\end{theorem}

Now we are ready to prove Theorem \ref{thm:sturmcompar}.

\begin{proof}[Proof of Theorem \ref{thm:sturmcompar}]
    By hypothesis, we have $X_2(K) = 0$. Then, Proposition~\ref{eq: prop: Jacobi} yields
    \[
        R(I) = J \quad \text{and} \quad IK + R(J) = 0.
    \]
    In particular, along a Reeb flow line $\varphi(t)$, the smooth function $I(t) = I \circ \varphi(t)$ satisfies the ODE
    \begin{equation}\label{eq:sturmode}
        \ddot{I}(t) + K(t) I(t) = 0.
    \end{equation}
    This ODE is of the same type as those appearing in the Sturm comparison theorem. We apply the latter by comparing $K(t)$ with $\inf K \circ \varphi$ and $\sup K \circ \varphi$. For the remainder of the proof, denote
    \[
        m = \inf K \circ \varphi \quad \text{and} \quad  M = \sup K \circ \varphi.
    \]
Since $K(t) \ge m$ for all $t$ and $y(t) = \sin(\sqrt{m}t)$ is a solution of
    \[
        \ddot{y}(t) + m\, y(t) = 0,
    \]
    the Sturm comparison theorem yields that either $I(t) = \sin(\sqrt{m}t)$ or else $I$ has a zero between any two consecutive zeros of $y(t) = \sin(\sqrt{m}t)$. In particular, the set $Z = \{t \in \R ; \ I(t) = 0\}$ of zeros of $I$ has infinitely many elements. Since $I$ satisfies the ODE \eqref{eq:sturmode} and is not identically zero, $Z$ is discrete. Moreover, any two consecutive zeros $t_1,t_2 \in Z$ must satisfy
    \[
        |t_1 - t_2| < \frac{2\pi}{\sqrt{m}}.
    \]
    This proves the first part of the theorem.

    For the second part, note that $M \ge K(t)$ for all $t$, and $y(t) = \sin(\sqrt{M}t)$ solves
    \[
        \ddot{y}(t) + M\, y(t) = 0.
    \]
    In this case, the Sturm comparison theorem yields that either $I(t) = \sin(\sqrt{M}t)$ or else there exists a zero of $y(t) = \sin(\sqrt{M}t)$ between any two consecutive zeros of $I(t)$. If $I(t) = \sin(\sqrt{M}t)$ we have $2\pi/\sqrt{M}$-periodicity and then we are done.

    Suppose that $I(t_1) = I(t_2)= \cdots = I(t_n)$ with $t_1<t_2<\cdots<t_n$. Since $\dot{I}(t) = J(t)$ and $I(t)$ is not identically zero, the function $I$ must have two zeros in each interval $(t_j, t_{j+1})$. In particular, there exists a zero of $y(t) = \sin(\sqrt{M}t)$ in each interval $(t_j,t_{j+1})$. In this case, we have $(n-1)$ such intervals and thus $(n-1)$ zeros of $y(t)$ in $(t_1,t_n)$. The result follows from the fact that the distance between two consectuve zeros of $y$ is $\pi/\sqrt{M}$.

Finally, if $\gamma \colon \R/T\Z \to Y$ is a Reeb orbit, we have $I(t) = I(t+T) = \cdots = I(t+nT)$ for $I(t) = I\circ \gamma(t)$ and any natural number $n$. Therefore, using the previous paragraph
$$nT > \frac{(n-1)\pi}{\sqrt{M}},$$
holds for any $n \in \mathbb{N}$ and this implies the desired upper bound.
\end{proof}

\subsection{Estimating action via CZ index}\label{sec: CZ_index}
\subsubsection{Conley--Zehnder index}
In this section, we recall the definition of the Conley--Zehnder index of a Reeb orbit but just in the particular case where $(Y,\lambda)$ assume a contact canonical frame $\{R,X_1,X_2\}$.

Let $\varphi \colon \R \times Y \to Y$ denote the Reeb flow on $(Y, \lambda)$ and $\gamma \colon [0,T] \to Y$ be a Reeb trajectory. The linearized Reeb flow defines a symplectic linear map
$$d\varphi_t(\gamma(0)) \colon \xi_{\gamma(0)} \to \xi_{\gamma(t)}.$$
Using the global symplectic\footnote{i.e. $\psi$ preserves the symplectic structure: $\psi^*d\lambda = \omega_0$, for the standard symplectic form $\omega_0$ on $\R^2$.} trivialization
\begin{align*}
    \psi \colon \R^2 \times Y &\to \xi \nonumber \\ (a,b,p) &\mapsto bX_1(p) + aX_2(p)
\end{align*}
we obtain a symplectic matrices smooth path $\Phi(t), t \in [0,T]$, defined by
\begin{equation}\label{eq:symppath}
    \Phi(t) = \psi(\cdot,\cdot,\gamma(t))^{-1} \circ d\varphi_t(\gamma(0)) \circ \psi(\cdot,\cdot, \gamma(0)).
\end{equation}
Note that $\Phi$ starts at $\Phi(0) = I_{2\times 2}$, the $2 \times 2$ identity matrix. Assume that $\gamma$ is a (closed) Reeb orbit, i.e., $\gamma(0) = \gamma(T)$. The Conley--Zehnder index of $\gamma$ with respect to the trivialization $\psi$ is, by definition, the Maslov index of the path $\Phi$. We follow the definition given in \cite[Section 3]{hofer1998dynamics}.

The symplectic trivialization $\psi$ identifies the standard complex structure
$$J_0 = \begin{bmatrix}
    0 & -1 \\ 1 & 0
\end{bmatrix}$$
on $\R^2 \cong \C$ with the almost complex structure $\mathbb{J}$ on $\xi$, see \eqref{def: almostcpx}.

Set $A(t) = -J_0\dot{\Phi}(t)\Phi(t)^{-1}$. Then $A(t)$ is a smooth path of symmetric matrices and $\Phi$ solves the ODE
\begin{equation}\label{eq:odecz}
    \begin{cases}
        \dot{\Phi} = J_0A\Phi \\
        \Phi(0) = I,
    \end{cases}
\end{equation}
for $t\in [0,T]$. It follows from \eqref{eq:symppath} and Lemma \ref{lemma:ode} that given $(a,b) \in \R^2$, we have
\begin{equation}\label{eq:phi(t)}
    \Phi(t)(a,b) = (y(t),x(t)),
\end{equation}
where $x(0) = b$, $y(0) = a$, $\dot{x} = y$ and $\dot{y} = -Kx$. Moreover, we have
\begin{equation}\label{eq:dotphi}
    \dot{\Phi}(t)(a,b) = \frac{d}{dt}(\Phi(t)(a,b)) = (\dot{y}(t),\dot{x}(t)) = (-Kx(t),y(t)).
\end{equation}
Using \eqref{eq:phi(t)} and \eqref{eq:dotphi} it is simple to check that for
$$J_0A = \begin{bmatrix}
    0 & -K \\ 1 & 0
\end{bmatrix},$$
$\Phi$ solves the ODE \eqref{eq:odecz}.
Since $J_0^{-1} = -J_0$, we find
\begin{equation}\label{eq:matrixA}
    A = \begin{bmatrix}
    1 & 0 \\ 0 & K
\end{bmatrix}.
\end{equation}

Consider the \emph{asymptotic operator} $L_A$ defined by
$$L_A(v) = -J_0 \dot{v} - A(t)v$$
in the Hilbert space of $T$-periodic maps $H^1(\R/T\Z, \R^2)$. It follows that the spectrum $\sigma(L_A)$ consists of real eigenvalues with multiplicity at most $2$ and is unbounded. Let $v\neq 0$ be an eigenfunction of $L_A$ associated to an eigenvalue $\tau \in \R$. Then
$$\tau v(t) = -J_0\dot{v}(t) - A(t)v(t),$$
and $v(0) = v(T)$ hold. Because $v$ does not vanish, there exists a smooth angle function $\theta(t)$ such that $v(t) = \vert v(t) \vert e^{2\pi i\theta(t)}$ for $t \in [0,T]$. We define the winding number\footnote{We note that eigenfunctions of the same eigenspace have the same winding number.}
$$\Delta(\tau,A) = \theta(T) - \theta(0).$$
From \cite[Lemma 3.6]{hofer1995properties}, the eigenvalues in $\sigma(L_A)$ come in pairs: for each $k \in \Z$, there exists exactly two eigenvalues $\tau,\tau'$ such that $\Delta(\tau,A) = \Delta(\tau',A) = k$. Therefore, we label the eigenvalues using the integers in such a way that $\tau_k \leq \tau_j$ for $k \leq j$ and $\Delta(\tau_k,A) = \left \lfloor \frac{k}{2} \right \rfloor$. In this case, the Conley--Zehnder index $CZ(\gamma)$ can be defined by
$$CZ(\gamma) = \mu(\Phi) = \max\{k \mid \tau_k <0\} \in \Z.$$

\subsubsection{Action and CZ index}
\begin{proof}[Proof of Theorem \ref{thm:estimate}]
    Let $\gamma \colon \R/T\Z \to Y$ be a closed Reeb orbit on $(Y,\lambda)$. By the discussion above, we know that the symplectic matrices path $\Phi$ satisfies the linear equation $\dot{\Phi} = J_0A\Phi$ for $A$ given by $\eqref{eq:matrixA}$. Let $v$ be an eigenvector of the asymptotic operator $L_A$ with eigenvalue $\tau$. So we have
    $$\tau v(t) = L_A(v) = -J_0\dot{v}(t)-A(t)v(t),$$
   and $v(0) = v(T)$. For $w(t) = v(t)/\vert v(t) \vert$, we obtain
   \begin{equation}\label{dotw}
\dot{w} = J_0(A+\tau I)w - w\langle w,J_0(A+\tau I) w \rangle.
   \end{equation}
   Writing $w(t) = e^{2\pi i \theta(t)}$, we get $2\pi\dot{\theta}J_0w = \dot{w}$. Then, we have
   \begin{equation}\label{2pitheta}
   \begin{aligned}
       2\pi\dot{\theta} &= \langle 2\pi \dot{\theta}J_0w,J_0w\rangle = \langle \dot{w}, J_0w \rangle \\ &= \langle J_0(A+\tau I)w - w\langle w,J_0(A+\tau I) w \rangle, J_0w \rangle \\ &= \langle J_0(A+\tau I)w, J_0w\rangle  \\ &= \langle (A+\tau I)w, w\rangle,
       \end{aligned}
   \end{equation}
   where we use \eqref{dotw} and the fact that $J_0$ is an isometry for the Euclidean inner product. Now using $\eqref{eq:matrixA}$,
   $$\langle (A+\tau I)w, w\rangle = (1+\tau)w_1^2 + (K+\tau)w_2^2,$$
   for $w(t) = (w_1(t),w_2(t))$. From \eqref{2pitheta}, we get
   $$\Delta(\tau,A) = \theta(T) - \theta(0) = \int_0^T \dot{\theta} \ dt = \frac{1}{2\pi}\int_0^T(1+\tau)w_1^2 + (K+\tau)w_2^2 \ dt.$$
   In particular, if $K \circ \gamma \geq 1$, we have
   \begin{equation}\label{estimategeq}
       \left \lfloor\frac{k}{2} \right \rfloor = \Delta(\tau_k,A) \geq \frac{1}{2\pi}\int_0^T (1+\tau_k)(w_1^2+w_2^2) = \frac{T}{2\pi}(1+\tau_k).
   \end{equation}
   On the other hand, if $K \circ \gamma \leq 1$, we get
   \begin{equation}\label{estimateleq}
              \left \lfloor\frac{k}{2} \right \rfloor = \Delta(\tau_k,A) \leq \frac{1}{2\pi}\int_0^T (1+\tau_k)(w_1^2+w_2^2) = \frac{T}{2\pi}(1+\tau_k).
   \end{equation}
   The desired estimates now follow from \eqref{estimategeq}, \eqref{estimateleq}, the fact that $\mathcal{A}(\gamma) = \int_\gamma \lambda = T$ and the definition of Conley--Zehnder index:
   $$CZ(\gamma) = \max\{k \mid \ \tau_k<0\}.$$
\end{proof}

\bibliographystyle{alpha}
\bibliography{biblio}
\end{document}